\documentclass[letterpaper, 11 pt]{article}
\usepackage{fullpage,amsthm, amsmath, amssymb, amsfonts}
\usepackage{graphicx}

\newcommand{\ZZ}{\mathbb{Z}}

\newcommand{\RR}{\mathbb{R}}

\newcommand{\newword}[1]{\textbf{\emph{#1}}}

\DeclareMathOperator{\CvxH}{ConvexHull}

\newcommand{\I}{\mathcal{I}}

\newtheorem{conj}{Conjecture}[section]

\newtheorem{theorem}[conj]{Theorem}

\newtheorem{proposition}[conj]{Proposition}

\newtheorem{lemma}[conj]{Lemma}

\newtheorem{corollary}[conj]{Corollary}

\theoremstyle{definition}
\newtheorem{definition}[conj]{Definition}

\begin{document}
 
\title{Semi-polytope decomposition of a Generalized permutohedron}

\author{SuHo Oh}




\date{}
\maketitle

\abstract{In this short note we show explicitly how to decompose a generalized permutohedron into semi-polytopes.}

\section{Introduction}

Given a polytope, assume we have disjoint open cells whose closures sum up to be the entire polytope. A question of naturally assigning each of the remaining points (possibly in multiple closures) to a cell has appeared in \cite{BBY} for studying regular matroids and zonotopes and in \cite{Oh-hvec} for studying h-vectors and $Q$-polytopes. In other words, we are trying to determine ownership of lattice points on boundaries of multiple polytopes. In this note, we study a more general case of doing the same for a \newword{Generalized permutohedron}, a polytope that can be obtained by deforming the usual permutohedron. We will show explicitly how to construct a semi-polytope decomposition of a trimmed generalized permutohedron.


\section{Generalized permutohedron $P_G$ and its fine mixed subdivision}

Let $\Delta_{[n]} = \CvxH(e_1,\ldots,e_n)$ be the standard coordinate simplex in $\RR^n$. For a subset $I \subset [n]$, let $\Delta_I = \CvxH(e_i | i \in I)$ denote the face of $\Delta_{[n]}$. Let $G \subseteq K_{m,n}$ be a bipartite graph with no isolated vertices. Label the vertices of $G$ by $1,\ldots,m,\bar{1},\ldots,\bar{n}$ and call $1,\ldots,m$ the \newword{left vertices} and $\bar{1},\ldots,\bar{n}$ the right vertices. We identify the barred indices with usual non-barred cases when it is clear we are dealing with the right vertices. For example when we write $\Delta_{\{\bar{1},\bar{3}\}}$ we think of it as $\Delta_{\{1,3\}}$. We associate this graph with the collection $\I_{G}$ of subsets $I_1,\ldots,I_m \subseteq [n]$ such that $j \in I_i$ if and only if $(i,\bar{j})$ is an edge of $G$. Let us define the polytope $P_G(y_1,\ldots,y_m)$ as:
$$P_G(y_1,\ldots,y_m) := y_1 \Delta_{I_1} + \cdots + y_m \Delta_{I_m},$$
where $y_i$ are nonnegative integers. This lies on a hyperplane $\sum_{i \in [n]} x_i = \sum_{j \in [m]} y_j$. An example of a coordinate simplex $\Delta_{[3]}$, a bipartite graph $G$ and a generalized permutohedron $P_G(1,2,3)$ is given in Figure~\ref{fig:ex1}.

\begin{figure}[htbp]
	\begin{center}
	 		\includegraphics[width=0.9\textwidth]{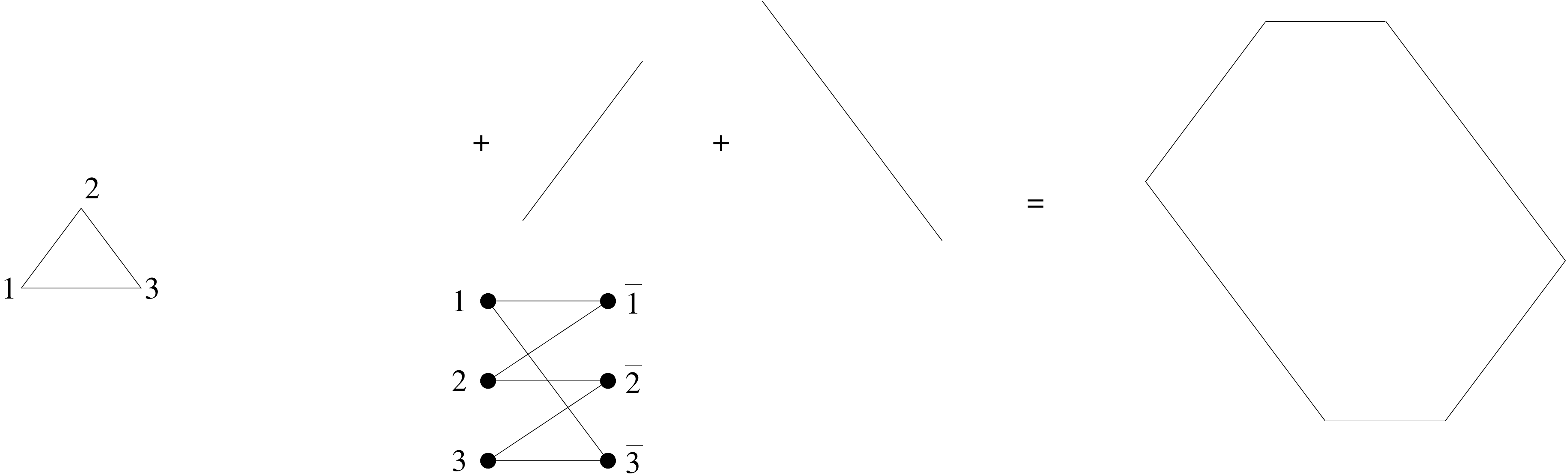}
		\caption{Example of a generalized permutohedron $P_G(1,2,3)$}
		\label{fig:ex1}
	\end{center}
\end{figure}

\begin{definition}[\cite{Postnikov01012009}, Definition 14.1]
Let $d$ be the dimension of the Minkowski sum $P_1 + \cdots + P_m$. A \newword{Minkowski cell} in this sum is a polytope $B_1 + \cdots + B_m$ of dimension $d$ where $B_i$ is the convex hull of some subset of vertices of $P_i$. A \newword{mixed subdivision} of the sum is the decomposition into union of Minkowski cells such that intersection of any two cells is their common face. A mixed subdivision is \newword{fine} if for all cells $B_1 + \cdots + B_m$, all $B_i$ are simplices and $\sum dim B_i = d$.
\end{definition}

All mixed subdivisions in our note, unless otherwise stated, will be referring to fine mixed subdivisions. We will use the word \newword{cell} to denote the Minkowski cells. Beware that our cells are all closed polytopes.

Fine Minkowski cells can be described by spanning trees of $G$. When we are looking at a fixed generalized permutohedron $P_G(y_1,\ldots,y_m)$, we will use $\prod_{J}$ to denote $y_1 \Delta_{J_1} + \cdots + y_m \Delta_{J_m}$ where $J = (J_1,\ldots,J_m)$. We say that $J$ is a tree if the associated bipartite graph is a tree.



\begin{lemma}[\cite{Postnikov01012009}, Lemma 14.7]
Each fine mixed cell in a mixed subdivision of $P_G(y_1,\cdots,y_m)$ has the form $\prod_{T}$ such that $T$ is a spanning tree of $G$.
\end{lemma}

An example of a fine mixed subdivision of the polytope considered in Figure~\ref{fig:ex1} is given in Figure~\ref{fig:ex2}.



\begin{figure}[htbp]
	\begin{center}
		\includegraphics[width=0.9\textwidth]{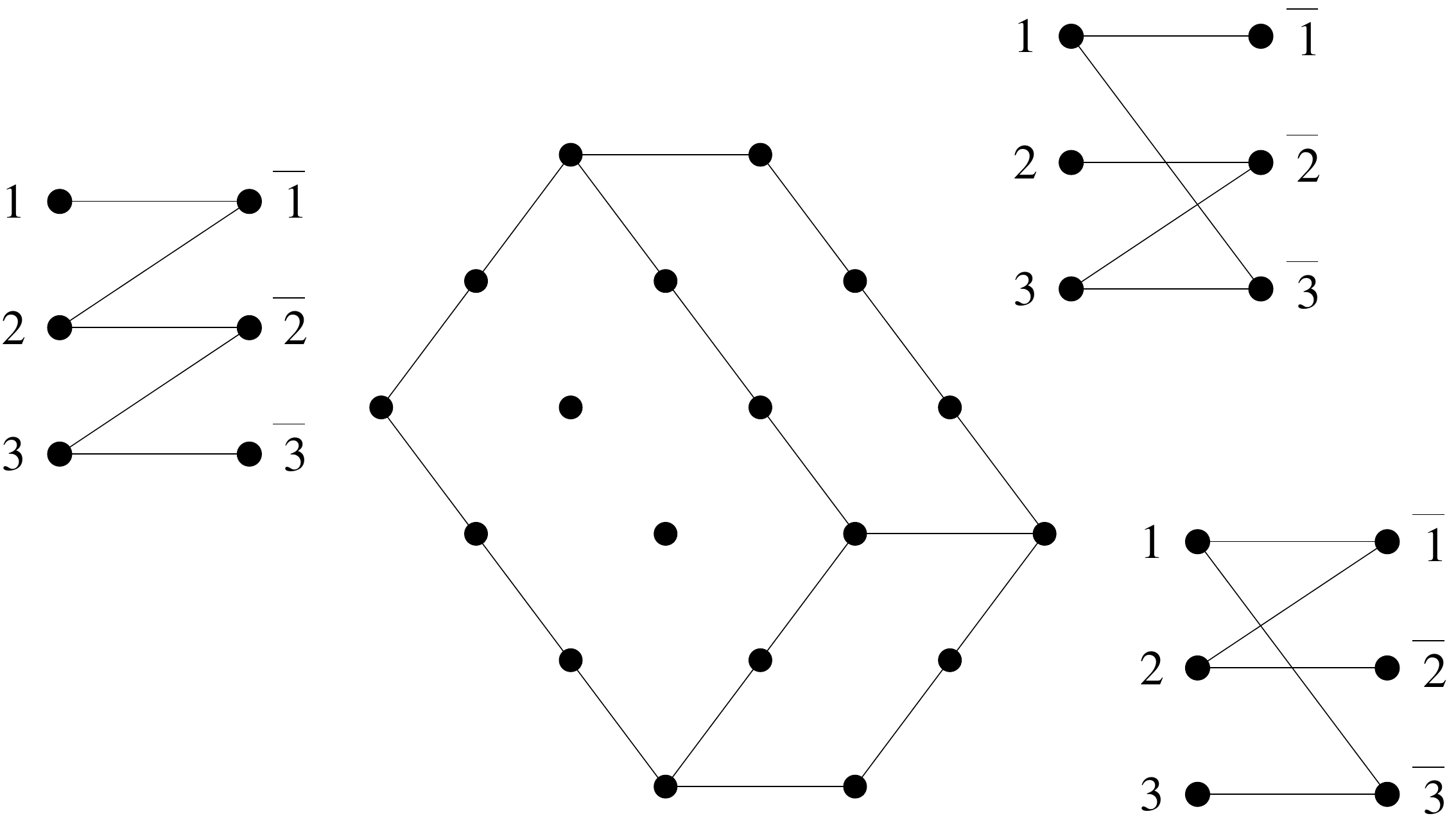}
		\caption{A fine mixed subdivision of $P_G(1,2,3)$.}
		\label{fig:ex2}
	\end{center}
\end{figure}

We can say a bit more about the lattice points in each $\prod_T$:

\begin{proposition}[\cite{Postnikov01012009}, Proposition 14.12]
\label{prop:ptsum}
Any lattice point of a fine Minkowski cell $\prod_T$ in $P_G(y_1,\cdots,y_m)$ is uniquely expressed (within $\prod_T$) as $p_1 + \cdots + p_m$ where $p_i$ is a lattice point in $y_i \Delta_{T_i}$.
\end{proposition}

\section{Semi-polytope decomposition}

A mixed subdivision of $P_G$ divides the polytope into cells. In this section, we show that from a mixed subdivision of $P_G$, one can obtain a way to decompose the set of lattice points of $P_G^{-}$.

\begin{definition}[\cite{Postnikov01012009}, Definition 11.2]
The \newword{trimmed generalized permutohedron} $P_G^{-}$ is defined as:
$$ P_G^{-}(y_1,\ldots,y_m) := \{x \in \RR^n|x+ \Delta_{[n]} \subseteq P_G\}. $$
\end{definition}

This is a more general class of polytopes than generalized permutohedra $P_G(y_1,\ldots,y_m)$. With a slight abuse of notation, we will let $I \setminus j$ stand for $I \setminus \{j\}$.


\begin{definition}[\cite{Postnikov01012009}, Theorem 11.3]
The coordinate \newword{semi-simplices} are defined as
$$\Delta_{I,j}^{*} = \Delta_{I} \setminus \Delta_{I \setminus j}$$
for $j \in I \subseteq [n]$.
\end{definition}

For each cell $\prod_T$, we are going to turn it into a \newword{semi-polytope} of the form $y_1 \Delta_{J_1,j_1}^* + \cdots + y_m \Delta_{J_m,j_m}^*$. This will involve deciding which cell takes ownership of the lattice points on several cells at the same time.

We denote the point $((m-1)c+\sum_i y_i,-c,\ldots,-c)$ for $c$ sufficiently large as $\infty_1$. For a facet of a polytope, we say that it is \newword{negative} if the defining hyperplane of the facet (inside the space $\sum_{i \in [n]} x_i = \sum_{j \in [m]} y_j$ which the polytope lies in) separates the point $\infty_1$ and the interior of the polytope. Otherwise, we say that it is \newword{positive}. We will say that a point of a polytope is \newword{good} if it is not on any of the positive facets of the polytope.

\begin{lemma}
\label{lem:iuni} 
Fix $T$, a spanning tree of $G \subseteq K_{m,n}$. Let $T_i$ be the set of neighbors of $i$. For each $i$ such that $\bar{1} \not \in T_i$, there exists a unique element $t_i$ in $T_i$ such that there exists a path to $\bar{1}$ not passing through $i$. 
\end{lemma}

\begin{proof}
There exists such an element since $T$ is a spanning tree of $T$. There cannot be more than one such element since otherwise, we get a cycle in $T$.
\end{proof}

In cases where $T_i$ does contain $\bar{1}$, we set $t_i$ to be $\bar{1}$.



\begin{lemma}
\label{lem:sign}
Let $\prod_T$ be a fine mixed cell. Removing the positive facets gives us $\sum_i \Delta^*_{T_i,t_i}$. 

\end{lemma}

To prove this, we first introduce a tool that will be useful for identifying which hyperplanes the facets lie on. Let $\prod_T$ be a fine mixed cell so $T$ a spanning tree. For any edge $e$ of $T$ that is not connected to a leaf on the left side, $T \setminus e$ has two components. Let $I_e$ denote the set of right vertices of a component that contains $\bar{1}$. Let $c_e$ be the sum of $y_i$'s for left vertices contained in that component. Notice that $I_e$ cannot be $[n]$ since otherwise $e$ would have a leaf as its left endpoint.

\begin{lemma}
\label{lem:temp}
Let $\prod_T$ be a fine mixed cell. For any edge $e$ of $T$ that is not connected to a leaf on the left side, $\prod_{T \setminus e}$ is a facet of $\prod_T$ that lies on $\sum_{j\in I_e} x_j = c_e$. If the right endpoint of $e$ is in $I_e$, then $\prod_T$ lies in half-space $\sum_{j\in I_e} x_j \geq c_e$. Otherwise it lies in $\sum_{j\in I_e} x_j \leq c_e$
\end{lemma}

\begin{proof}
The dimension difference between $\prod_T$ and $\prod_{T \setminus e}$ is at most one, and all endpoints of $\prod_{T \setminus e}$ lie on $\sum_{j\in I_e} x_j = c_e$. If the right endpoint of $e$ is in $I_e$, that means we can find a point $x$ using $e$ so that $\sum_{j\in I_e} x_j > c_e$. If not, that means we can find a point $x$ using $e$ so that $\sum_{j\in I_e} x_j < c_e$.
\end{proof}

\begin{proof}[Proof of Lemma~\ref{lem:sign}]

 If $\prod_{T \setminus e}$ is a positive facet of $\prod_T$, then Lemma~\ref{lem:temp} tells us that the right endpoint of $e = (i,\bar{j})$ is in $I_e$. From definition of $t_i$, we have $\bar{j} = t_i$. In other words we are removing sets of form $\Delta_{T_1} + \cdots + \Delta_{T_i \setminus t_i} + \cdots + \Delta_{T_m}$. At the end we end up with $\sum_i (\Delta_{T_i} \setminus \Delta_{T_i \setminus t_i})$.

\end{proof}

\begin{figure}[htbp]
	\begin{center}
		\includegraphics[width=0.9\textwidth]{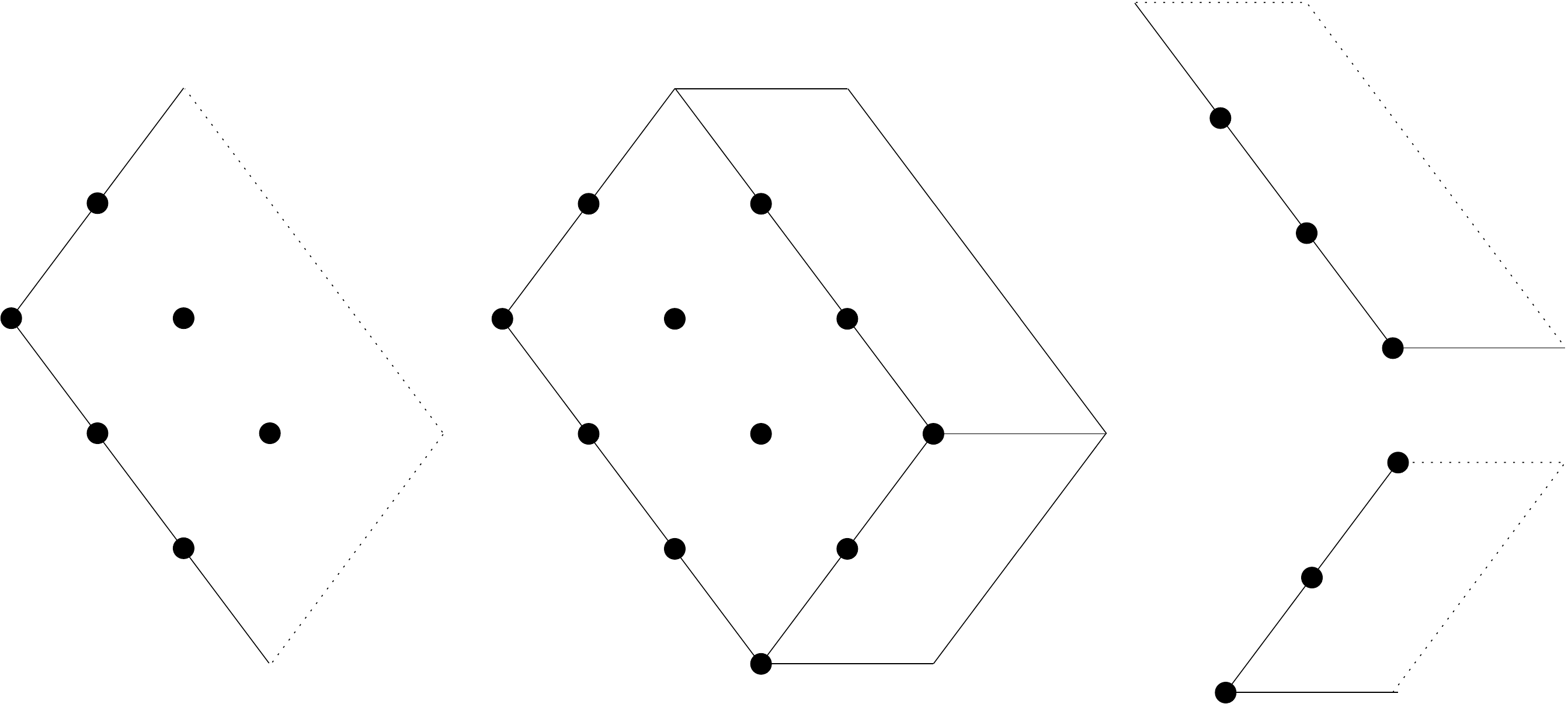}
		\caption{A semi-polytope decomposition of $P_G^{-}(1,2,3)$.}
		\label{fig:ex3}
	\end{center}
\end{figure}

Let $\prod_T$ and $\prod_T'$ be two cells inside a mixed subdivision of $P_G$ that share a facet $F$. The sign of $F$ in $\prod_T$ and the sign of $F$ in $\prod_T'$ has to be different, since $\infty_1$ can be on exactly one side of the defining hyperplane of $F$. This implies that all lattice points of $P_G$ are good in at most one cell of $P_G$. 

\begin{lemma}
\label{lem:shift}
$p \in P_G^{-} \cap \ZZ^n$ if and only if $p + e_1$ is a good point of $P_G$.
\end{lemma}
\begin{proof}
Having $p \in P_G^{-} \cap \ZZ^n$ implies from definition that $p + e_i \in P_G$ for each $i \in [n]$. Assume for sake of contradiction $p+e_1$ is on some positive facet $x_I = c_I$ with $1 \in I$. Any $p+e_i$ is either on that facet or is on the same side as $\infty_1$. Hence $I = [n]$ and we get a contradiction.


Now look at the case when $p+e_1$ is a good point of $P_G$. Assume for the sake of contradiction that $p+e_j$ is not in $P_G$ for some $j \in [n]$. Then $p+e_1$ is on a facet of $P_G$, whose corresponding hyperplane is given by $x_I = c_I$ where $1 \in I$ and $j \not \in I$. This hyperplane separates the interior of $P_G$ with $p+e_j$. Since $j \not \in I$, the point $\infty_1$ has to be on opposite side of $p+e_j$. This is a positive facet and we get a contradiction.

\end{proof}

Combining what we have so far, we state the main result of our note on how to do the semi-polytope decomposition with an explicit way to obtain each of the semi-polytopes:
\begin{theorem}
\label{thm:spdecomp}
Identify the lattice points of $P_G^{-}$ with points not lying on any of the positive facets of $P_G$ via the map $p \rightarrow p+e_1$ (as in Lemma~\ref{lem:shift}). Pick any full mixed subdivision of $P_G$. For each cell $\prod_T$, construct a semi-polytope by $\sum_i \Delta^{*}_{T_i,t_i}$ (where $t_i$ is chosen as in Lemma~\ref{lem:iuni}). Then the (disjoint) union of the semi-polytopes is exactly $P_G$ minus the positive facets. Each lattice point of $P_G^{-}$ with the above identification is contained in exactly one semi-polytope.
\end{theorem}

An example of a semi-polytope decomposition of the generalized permutohedron considered in Figure~\ref{fig:ex1} and in Figure~\ref{fig:ex2} is given in Figure~\ref{fig:ex3}.

\section{Application to Erhart theory}

In this section we show how the semi-polytope decomposition can be used in Erhart theory as guided in \cite{Postnikov01012009}. Given any subgraph $T$ in $G$, define the \newword{left degree vector} $ld(T)=(d_1-1,\cdots,d_n-1)$ and the \newword{right degree vector} $rd(T) = (d_1'-1,\cdots,d_m'-1)$ where $d_i$ and $d_j'$ are the degree of the vertex $i$ and $\bar{j}$ respectively. The raising powers are defined as $(y)_a := y(y+1)\cdots(y+a-1)$ for $a \geq 1$ and $(y)_0 :=1$.

\begin{corollary}
\label{cor:main2}

Fix a fine mixed subdivision of $P_G(y_1,\ldots,y_m)$ where $y_i$'s are nonnegative integers. The number of lattice points in the trimmed generalized permutohedron $P_G^{-}(y_1,\ldots,y_m)$ equals $\sum_{(a_1,\ldots,a_m)} \prod_i \frac{(y_i)_{a_i}}{a_i!}$ where the sum is over all left-degree vectors of fine mixed cells inside the subdivision.

\end{corollary}
\begin{proof}

Obtain a semi-polytope decomposition as in Theorem~\ref{thm:spdecomp}. Then each lattice point of $P_G^{-}(y_1,\ldots,y_m)$ is in exactly one semi-polytope. The claim follows since the number of lattice points of a semi-polytope $y_1 \Delta_{T_1,t_1}^* + \cdots + y_m \Delta_{T_m,t_m}^*$ (thanks to Proposition~\ref{prop:ptsum}, different sum gives a different point) is given by $\prod_i \frac{(y_i)_{|T_i|-1}}{(|T_i|-1)!}$.

\end{proof}

The expression in Corollary~\ref{cor:main2} is called the \newword{Generalized Erhart polynomial} of $P_G^{-}(y_1,\ldots,y_m)$ by \cite{Postnikov01012009}. As foretold in \cite{Postnikov01012009}, Theorem~\ref{thm:spdecomp} gives us a pure counting proof of Theorem 11.3 of \cite{Postnikov01012009}.

\begin{definition}[\cite{Postnikov01012009}, Definition 9.2]
Let us say that a sequence of nonnegative integers $(a_1,\cdots,a_m)$ is a $G$-\newword{draconian sequence} if $\sum a_i = n-1$ and, for any subset $\{i_1 < \cdots < i_k\} \subseteq [m]$, we have $|I_{i_1} \cup \cdots \cup I_{i_k}| \geq a_{i_1} + \cdots + a_{i_k}+1$. 
\end{definition}

\begin{theorem}[\cite{Postnikov01012009}, Theorem 11.3]
\label{thm:main}
For nonnegative integers $y_1,\ldots,y_m$, the number of lattice points in the trimmed generalized permutohedron $P_G^{-}(y_1,\ldots,y_m)$ equals $\sum_{(a_1,\ldots,a_m)} \prod_i \frac{(y_i)_{a_i}}{a_i!}$, where the sum is over all $G$-draconian sequences $(a_1,\ldots,a_m)$.
\end{theorem}
\begin{proof}
Thanks to Corollary~\ref{cor:main2}, all we need to do is show that the set of $G$-draconian sequences is exactly the set of left-degree vectors of the cells inside a fine mixed subdivision of $P_G$. Lemma 14.9 of \cite{Postnikov01012009} tells us that the right degree vectors of the fine cells is exactly the set of lattice points of $P_{G^*}^-(1,\ldots,1)$ where $G^*$ is obtained from $G$ by switching left and right vertices. Then Lemma 11.7 of \cite{Postnikov01012009} tells us that the set of $G$-draconian sequences is exactly the set of lattice points of $P_{G^*}^-(1,\ldots,1)$. 
\end{proof}

This approach has an advantage that the generalized Erhart polynomial is obtained from a direct counting method, without using any comparison of formulas.

\bibliographystyle{plain}    
\bibliography{tom}        
 
\end{document}